\documentclass[11pt]{amsart}
\usepackage{amsfonts, amsthm, amsmath, amssymb}
\usepackage[all]{xy}
\usepackage[top=1in,bottom=1in,left=1in,right=1in]{geometry}
\usepackage{tikz}
\usepackage[pdftex]{hyperref}

\newcommand{\bC}{\mathbb{C}}

\newcommand{\bF}{\mathbb{F}}
\newcommand{\bG}{\mathbb{G}}

\newcommand{\bQ}{\mathbb{Q}}
\newcommand{\bR}{\mathbb{R}}

\newcommand{\bZ}{\mathbb{Z}}
\newcommand{\unit}{\mathbf{1}}

\newcommand{\fg}{\mathfrak{g}}

\newcommand{\uA}{\underline{A}}

\newcommand{\simto}{\overset{\sim}{\to}}

\newcommand{\Aut}{\operatorname{Aut}}
\newcommand{\uAut}{\operatorname{\underline{Aut}}}

\newcommand{\End}{\operatorname{End}}

\newcommand{\Hom}{\operatorname{Hom}}
\newcommand{\uHom}{\operatorname{\underline{Hom}}}
\newcommand{\id}{\operatorname{id}}
\newcommand{\Ind}{\operatorname{Ind}}

\newcommand{\Lie}{\operatorname{Lie}}

\newcommand{\rank}{\operatorname{rank}}

\newcommand{\Spec}{\operatorname{Spec}}

\newcommand{\Sym}{\operatorname{Sym}}

\newtheorem{thm}{Theorem}[subsection]

\newtheorem{prop}[thm]{Proposition}
\newtheorem{lem}[thm]{Lemma}
\newtheorem{cor}[thm]{Corollary}

\theoremstyle{definition}
\newtheorem{defn}[thm]{Definition}

\theoremstyle{remark}
\newtheorem{rem}[thm]{Remark}

\begin{document}

\title{Automorphism group schemes of lattice vertex algebras}
\author{Scott Carnahan, Hayate Kobayashi}

\begin{abstract}
Given a positive definite even lattice and a commutative ring, there is a standard construction of a lattice vertex algebra over the commutative ring, and it admits a natural grading by non-negative integers.  We describe the groups of automorphisms of these graded vertex algebras as affine group schemes, showing in particular that each is an extension of an explicitly described split reductive group of ADE type by the outer automorphism group of the lattice.
\end{abstract}

\maketitle

\tableofcontents

\subsection*{Introduction}

Lattice vertex operator algebras, or more precisely, some of the structures underlying them, were the original motivation for the development of vertex algebras in \cite{B86}, and a fundamental step toward the construction of the monster vertex operator algebra $V^\natural$ in \cite{FLM88}.  Physically, lattice vertex operator algebras describe the propagation of bosonic strings in spacetime tori, and they form a rich family of examples with many interesting symmetry groups.  In particular, the symmetry groups of the lattice vertex operator algebras attached to certain Niemeier lattices were essential to the proof that the monster simple group acts faithfully on $V^\natural$.  For these reasons, it is natural to study the symmetries of lattice vertex operator algebras in general.  The full symmetry group of an arbitrary lattice vertex operator algebra was described in \cite{DN98} for the case where the base ring is the field of complex numbers.  In this paper, we generalize their results to determine the groups of symmetries of lattice vertex operator algebras over arbitrary commutative rings.

There are several challenges to overcome when doing this generalization, but perhaps the most prominent is that, because we are working over general commutative rings, we are forced to use the language of group schemes instead of the traditional tools of complex Lie group theory. This linguistic barrier can be substantial for mathematicians whose work is not ``scheme-adjacent''. Fortunately, we do not have to wade far into the depths of the theory - we need sheaf-theoretic arguments on a few occasions, but we never need to leave the world of affine schemes.

The statement of Dong-Nagatomo's theorem is that if $L$ is a positive-definite even lattice, then the automorphism group of the lattice vertex operator algebra $V_{L,\bC}$ is a product of a certain Lie group $N$ with a finite group $O(\hat{L})$.  They also show that the intersection $N \cap O(\hat{L})$ contains the 2-torsion group $\Hom(L,\pm 1)$.  Our result (Theorem \ref{thm:main}) replaces $N$ with a split reductive algebraic group $G_L$ defined over $\bZ$, and the finite group $O(\hat{L})$ with a finite flat group scheme $O(\tilde{L})$.  It should be unsurprising that the complex analytification of $G_L$ is isomorphic to $N$, and that $O(\tilde{L})(\bC) \cong O(\hat{L})$, so we obtain Dong-Nagatomo's result by specialization.  We also determine the intersection $G_L \cap O(\tilde{L})$ precisely, identifying it with the ``Tits group'' of $G_L$, and use this to show that the quotient $\Aut V_L/G_L$ is isomorphic to the outer automorphism group $\Aut L/W_L$ of the lattice, i.e., the quotient of all orthogonal transformations by the subgroup generated by reflections.

In addition to these generalizations, we remove the hypothesis that $V_L$ has a conformal vector, instead considering homogeneous vertex algebra automorphisms.  This effectively means we allow ourselves to consider commutative rings $R$ where the determinant of $L$ is not invertible.  The fact that we still get a uniform description of the automorphism group is somewhat surprising, since $V_L$ is no longer a simple vertex algebra when $\det L = 0$ in $R$. In particular, there may be quotients with exceptionally large symmetry.  One concrete example is the $A_2$ lattice over $\bF_3$, whose vertex algebra has a simple quotient with automorphism group $G_2(3)$: see \cite{GL13}, Proposition 6.1 and Theorem 6.3 \footnote{We thank C. H. Lam for bringing this example to our attention.}.

Viewed from a suitable distance, our proof is more or less a translation of Dong-Nagatomo's proof, with some necessary adjustments. One such adjustment lies in the construction of $G_L$. Dong and Nagatomo construct $N$ by simply exponentiating the weight 1 Lie algebra of $V_{L,\bC}$, but reductive algebraic groups in general cannot be constructed by exponentiating a Lie algebra.  Instead, we construct $G_L$ by describing its torus $T = \uHom(L, \bG_m)$ separately from its one-dimensional root subgroups $U_\alpha$, then showing that these subgroups satisfy the appropriate relations.

The main sticking point for us was the conjugacy of tori: over an algebraically closed field, all Cartan subgroups of a smooth affine group scheme are conjugate to each other, but this doesn't hold for general commutative rings.  The standard example is $SL_2(\bR)$, in which the diagonal torus is not conjugate to $SO(2)$. Since the maximal tori we consider are split, there are fairly strong replacement results: Proposition 6.2.11 in \cite{C14} asserts that if $G$ is split reductive over a nonempty connected scheme $S$ with trivial Picard group, then all split maximal tori are conjugate.  It may be possible to turn this result into a solution, but we found it more straightforward to use \'etale-local conjugacy of all maximal tori. Using the \'etale sheaf property of affine schemes, we find that this is sufficient to show that all automorphisms are covered.

\subsection*{Overview of the paper}

In section 1, we introduce vertex algebras and the lattice vertex algebra construction over commutative rings.  We show that if $\det L$ is invertible in a commutative ring $R$, then the standard conformal vector over $\bQ$ extends to the lattice vertex algebra over $R$.

In section 2, we introduce some basic theory of reductive group schemes, and define the group $G_L$.  It is a split reductive group over $\bZ$, and it acts faithfully on $V_L$.

In section 3, we define the finite flat group scheme $O(\tilde{L})$, and show that it acts faithfully on $V_L$.

In section 4, we analyze the automorphism group scheme of $V_L$, and prove the main theorem.

In section 5, we propose some generalizations.

\subsection*{Acknowledgements}

S. C. would like to thank Robert Griess and Ching-Hung Lam for helpful discussions, and Arturo Pianzola for helpful comments on an earlier version of this paper. This work was supported by JSPS Kakenhi Grant-in-Aid for Young Scientists (B) 17K14152, and JSPS Kakenhi Grant-in-Aid for Scientific Research (C) 22K03264

\section{Review of vertex algebras over commutative rings}

\subsection{Vertex algebras}

\begin{defn}
A \textbf{vertex algebra} over a commutative ring $R$ is an $R$-module $V$ equipped with a distinguished vector $\unit \in V$ and a multiplication map $Y: V \otimes_R V \to V((z))$, written as $a \otimes b \mapsto Y(a,z)b = \sum_{n \in \bZ} a_n b z^{-n-1}$, satisfying the following conditions:
\begin{enumerate}
\item $Y(\unit,z) = id_V z^0$ and $Y(a,z)\unit \in a + zV[[z]]$.
\item For any $r,s,t \in \bZ$, and any $u,v,w \in V$, the `Borcherds identity' holds:
\[ \sum_{i \geq 0} \binom{r}{i} (u_{t+i} v)_{r+s-i} w = \sum_{i \geq 0} (-1)^i \binom{t}{i} (u_{r+t-i}(v_{s+i}w) - (-1)^t v_{s+t-i}(u_{r+i} w)) \] 
\end{enumerate}
A vertex algebra homomorphism $(V, \unit_V, Y_V) \to (W, \unit_W, Y_W)$ is an $R$-module homomorphism $\phi: V \to W$ satisfying $\phi(\unit_V) = \unit_W$ and $\phi(u_n v) = \phi(u)_n \phi(v)$ for all $u,v \in V$ and all $n \in \bZ$.  Given a vertex algebra $V$ over $R$ and a vertex algebra $W$ over $S$, we say that $V$ is an $R$-form of $W$ with respect to a ring homomorphism $\phi: R \to S$ if $V \otimes_{R,\phi} S$ is isomorphic to $W$ as a vertex algebra over $S$.  If $\phi$ is implicitly fixed, we will simply say that $V$ is an $R$-form of $W$.  
\end{defn}

\begin{rem}
This definition of vertex algebra is equivalent to many others.  For example, one may replace the Borcherds identity with its generating function version, the ``Jacobi identity'':
\[ \begin{aligned}
x^{-1} \delta\left(\frac{y-z}{x}\right) &Y(a,y) Y(b,z) - x^{-1} \delta\left(\frac{z-y}{-x}\right) Y(b,z) Y(a,y) =  \\
&= z^{-1} \delta\left(\frac{y-x}{z}\right) Y(Y(a,x)b,z),
\end{aligned} \]
where $\delta(z) = \sum_{n \in \bZ} z^n$, and $\delta\left(\frac{y-z}{x}\right)$ is expanded as a formal power series with non-negative powers of $z$, i.e., as $\sum_{n \in \bZ, m \in \bZ_{\geq 0}} (-1)^m \binom{n}{m} x^{-n}y^{n-m}z^m$.  Alternatively, one may replace it with the combination of one of:
\begin{itemize}
\item (commutator formula) For any $u,v \in V$, and any $m,n \in \bZ$,
\[ [u_m, v_n] = \sum_{k \geq 0} \binom{m}{k} (u_k v)_{m+n-k}, \]
or its equivalent $[u_m,Y(v,z)] = \sum_{k \geq 0} \binom{m}{k} Y(u_k v,z) z^{m-k}$
\item (locality) For any $u,v \in V$, there exists $N \geq 0$ such that $(x-y)^N [Y(u,x),Y(v,y)] = 0$.
\item (skew-symmetry) For any $u,v \in V$ and any $n \in \bZ$, $v_n u = (-1)^{n+1} \sum_{i \geq 0} (-1)^i L_{-1}^{(i)}(u_{n+i}v)$. 
\end{itemize}
and one of
\begin{itemize}
\item (associativity) $(u_t v)_s w = \sum_{i \geq 0} (-1)^i \binom{t}{i} (u_{t-i}(v_{s+i}w) - (-1)^t v_{s+t-i}(u_i w))$
\item (weak associativity) For any $u,v, w \in V$, $(x + y)^NY(Y(u,x)v, y)w = (x+y)^N Y(u,x + y)Y(v,y)w$ for sufficiently large $N$.
\item (translation-covariance) $[T^{(m)}, Y(u,z)] = \sum_{i=1}^m \partial^{(i)}Y(u,z) T^{(m-i)}$, where $T^{(m)} : V \to V$ is the linear map $u \mapsto u_{-n-1}\unit$.
\end{itemize}
Proofs can be easily translated from the usual equivalences over $\bC$ found in, e.g., \cite{LL04} and \cite{MN97}.  We also note that the condition $Y(\unit,z) = id_V z^0$ is redundant: see Theorem 2.7 in \cite{M17} for a surprisingly long proof.
\end{rem}

\begin{defn}
A $\bZ$\textbf{-graded vertex algebra} over a commutative ring $R$ is a $\bZ$-graded $R$-module $V = \bigoplus_{n \in \bZ} V_n$, with a vertex algebra structure $(\unit, Y)$ on the underlying $R$-module $V$, such that the following grading-compatibility condition is satisfied:
\begin{itemize}
\item For any $k,m,n \in \bZ$, and vectors $u \in V_k, v \in V_m$, we have $u_n v \in V_{k+m-n-1}$.
\end{itemize}
A homomorphism of $\bZ$-graded vertex algebras is a vertex algebra homomorphism that preserves degrees.
\end{defn}
                                             
We note that since $\unit_n = \delta_{n,-1}\id_V$, we necessarily have $\unit \in V_0$ (see \cite{M17} Lemma 7.9 for details).

\begin{defn}
A \textbf{vertex operator algebra} over a commutative ring $R$ with \textbf{central half-charge} $c \in R$ is a $\bZ$-graded vertex algebra $V = \bigoplus_{n \in \bZ} V_n$, together with a conformal vector $\omega \in V$ with coefficients $Y(\omega,z) = \sum_{i \in \bZ} L_i z^{-i-2}$.  These data must satisfy the following:
\begin{enumerate}
\item $L_0$ acts diagonalizably with integer eigenvalues.
\item For all $n \in \bZ$, $v \in V_n$, we have $L_0 v = nv$.
\item There is some $N \in \bZ$ such that $V_n = 0$ for $n < N$.
\item Each $V_n$ is a finite type projective $R$-module.
\item For all $u \in V$, $L_{-1}u = u_{-2} \unit$
\item For all $m,n \in \bZ$, $[L_m, L_n] = (m-n)L_{m+n} + \binom{m+1}{3}c \id_V$.
\end{enumerate}
A homomorphism of vertex operator algebras over $R$ is a vertex algebra homomorphism that respects both the $\bZ$-grading, and the conformal vectors.
\end{defn}

When $2$ is invertible in $R$, it is common to use the ``central charge'', which is $2$ times the central half-charge.

\begin{defn}
Given a $\bZ$-graded vertex algebra $V$ over a commutative ring $R$, the automorphism functor $\uAut_{V/R}$ takes any commutative ring $S$ over $R$ to the group $\uAut_{V/R}(S)$ of automorphisms of $V \otimes_R S$ as a $\bZ$-graded vertex algebra over $S$.
\end{defn}

\begin{rem}
We have multiple choices for defining $\uAut_{V/R}$: we could alternatively use vertex algebra automorphisms, or if $V$ is a vertex operator algebra, we could consider automorphisms that fix $\omega$.  We have chosen to require degree-preserving automorphisms, because the examples we consider have a Lie algebra structure on the weight 1 space, and this gives us strong control over the infinitesimal automorphisms.  However, we have chosen to not require conformal vectors, because we would like to construct a group scheme over $\bZ$, and the examples we consider do not always have conformal vectors over $\bZ$.
\end{rem}

\begin{lem} \label{lem:faithfully-flat-sheaf}
Let $V$ be a $\bZ$-graded vertex algebra over a commutative ring $R$.  Then, $\uAut_{V/R}$ is a group sheaf for the faithfully flat topology.
\end{lem}
\begin{proof}
First, we show the sheaf property of the group-valued functor $\uAut_{V/R}$.  For any faithfully flat ring homomorphism $R \to S$, base change of automorphisms induces a diagram of group homomorphisms
\[ \xymatrix{\uAut_{V/R}(R) \ar[r]^f & \uAut_{V/R}(S) \ar@/^/[r]^-{g} \ar@/_/[r]_-{h} & \uAut_{V/R}(S \otimes_R S) } \]
Here, $g$ and $h$ take an $S$-automorphism $x$ of $V \otimes_R S$ and produce the automorphisms $g(x)$ and $h(x)$ of $V \otimes_R S \otimes_R S$ defined as follows:
\[ \begin{aligned}
g(x)(v \otimes s \otimes s') &= \sum_i v_i \otimes s_i \otimes s' \\
h(x)(v \otimes s \otimes s') &= \sum_j v_j \otimes s \otimes s'_j
\end{aligned} \]
where we write $x(v \otimes s) = \sum_i v_i \otimes s_i$ and $x(v \otimes s') = \sum_j v_j \otimes s'_j$.

We claim that this diagram is an equalizer, that is, $f$ is injective and the image of $f$ is equal to $\{ x \in \uAut_{V/R}(S) | g(x) = h(x) \}$.  Let $x,y \in \uAut_{V/R}(R)$ satisfy $f(x) = f(y)$.  Then for any $v \otimes s \in V \otimes_R S$ we have $x(v) \otimes s = y(v) \otimes s$.  Since $R \to S$ is injective, this implies $x(v) = y(v)$, so $f$ is injective.

For any $x\in \uAut_{V/R}(R)$ and any $v \otimes s \in V \otimes_R S$, $f(x)(v \otimes s) = x(v) \otimes s$, and both $g(f(x))$ and $h(f(x))$ take $v \otimes s \otimes s'$ to $x(v) \otimes s \otimes s'$.  Thus, the image of $f$ is contained in $\{ y \in \uAut_{V/R}(S) | g(y) = h(y) \}$.

We now prove the opposite containment.  Let $x \in \uAut_{V/R}(S)$ satisfy $g(x) = h(x)$, so with our earlier notation, $\sum_i v_i \otimes s_i \otimes s' = \sum_j v_j \otimes s \otimes s'_j$.  These elements lie in the intersection of $V \otimes_R S \otimes_R Rs'$ and $V \otimes_R Rs \otimes_R S$ in $V \otimes_R S \otimes_R S$, which is $V \otimes_R Rs \otimes_R Rs'$.  By shifting scalars in $R$ to the $V$ factor, we see that both $g(x)(v \otimes s \otimes s')$ and $h(x)(v \otimes s \otimes s')$ have the form $v' \otimes s \otimes s'$ for some $v' \in V$, so $x(v \otimes s) = v' \otimes s$.  We conclude that $x$ has the form $x' \otimes \id$ for some $x' \in \uAut_{V/R}(R)$.
\end{proof}

\begin{lem} \label{lem:represented-by-affine}
 Let $V$ be a $\bZ$-graded vertex algebra over a commutative ring $R$.  If $V$ is finitely generated, then $\uAut_{V/R}$ is represented by a finite type affine group scheme over $R$.
\end{lem}
\begin{proof}
The assumption that graded pieces are finite type and projective implies any finite generating set lies in some finite type projective module $U$ that is invariant under automorphisms, so we obtain an embedding in a product of general linear groups.  It remains to show that the condition ``vertex algebra multiplication is respected'' is Zariski closed, so $\uAut_{V/R}$ is represented by a closed affine group subscheme of some general linear group.

To prove the Zariski closed condition, we adapt the following variant of the argument in section 2 of \cite{DG01}. Let $U$ be a finite dimensional sum of homogeneous pieces of $V$ that contains a full generating set. For each $n$-tuple $(a_1,\ldots,a_n)$ of integers, we obtain a linear map $U^{\otimes n + 1} \to V$ by composing the corresponding vertex operators:
\[ u^1 \otimes \cdots \otimes u_{n+1} \mapsto u^1_{a_1}u^2_{a_2}\cdots u^n_{a_n}u^{n+1} \]
For a homogeneous linear transformation $g$ on $U$ to be the restriction of an automorphism of $V$, it is necessary and sufficient that all of these products be preserved, for all $n \in \bZ_{\geq 0}$.  That is, for any $v \in V$, and any two presentations of $v$ as a finite sum of elements of the form $u^1_{a_1}u^2_{a_2}\cdots u^n_{a_n}u^{n+1}$, we require that the corresponding finite sums of elements of the form $g(u^1)_{a_1} g(u^2)_{a_2} \cdots g(u^n)_{a_n}g(u^{n+1})$ coincide. Each pair of presentations of a vector $v$ thus yields a polynomial relation on the coordinates of $g$ (viewed as a linear transformation on $U$), and if $g$ satisfies all such relations, then it is the restriction of a uniquely defined automorphism of $V$.  These polynomial relations therefore describe Zariski-closed subsets of $GL(U)$.
\end{proof}

\begin{defn}
Let $V$ be a finitely generated $\bZ$-graded vertex algebra over a commutative ring $R$.  We write $\Aut_{V/R}$ for the affine group scheme that represents the functor $\uAut_{V/R}$.
\end{defn}

\subsection{Lattice vertex algebras} \label{sec:lattice}

We review the lattice vertex algebra construction over $\bZ$, outlined in \cite{B86}, with more refined explanations in \cite{P92}, \cite{DG12}, and \cite{M14}.

\begin{defn}
Let $L$ be an integer lattice, i.e., a finitely generated free abelian group equipped with a nondegenerate integer-valued symmetric bilinear form.  We define the \textbf{rational Heisenberg Lie algebra} $H_L$ on $L$ to be the $\bQ$-vector space $(L \otimes \bQ)[t,t^{-1}] \oplus \bQ K$ with bracket given by $[t^i,t^j] = i \delta_{i+j,0}$ and $[t_i,K] = 0$.  This is a central extension of the abelian Lie algebra $(L \otimes \bQ)[t,t^{-1}]$ by the one dimensional Lie algebra $\bQ K$.  We write $H_L^{\geq 0}$ for the abelian subalgebra $(L \otimes \bQ)[t] \oplus \bC K$, and for any $a \in L \otimes \bQ$ and $r \in \bQ$, we write $\bQ_{a,r}$ for the 1-dimensional representation on which the subalgebra $t(L \otimes \bQ)[t]$ acts trivially, $(L \otimes \bQ)$ acts by inner product with $a$, and $K$ acts by the scalar $r$.  Then, we write $\pi^L_{a,\bQ}$ for the induced module $\Ind_{H_L^{\geq 0}}^{H_L} \bQ_{a,1}$, and identify this space with $\Sym_{\bQ} (t^{-1}(L \otimes \bQ)[t^{-1}]) \otimes \bQ e_a$ for a distinguished vector $e_a$.  For any $b \in L \otimes \bQ$ and $n \in \bZ$, we define $b(n)$ to be the operator
\[ a_{i_1} t^{i_1} \cdots a_{i_k} t^{i_k} e_a \mapsto \begin{cases} bt^n a_{i_1}t^{i_1} \cdots a_{i_k} t^{i_k} & n < 0 \\ (b,a) a_{i_1} t^{i_1} \cdots a_{i_k} t^{i_k} e_a & n = 0 \\
\sum_{j | i_j = -n} b_{i_1}t^{i_1}\cdots b_{i_{j-1}}t^{i_{j-1}} (a,b_{i_j}) b_{i_{j+1}}t^{i_{j+1}} \cdots b_{i_k} t^{i_k} e_a & n > 0 \end{cases} \]
\end{defn}

\begin{thm} (\cite{B86}, \cite{FLM88})
There is a unique $\bZ$-graded $\bQ$-vertex algebra structure on $\pi^L_{0,\bQ}$, with identity $1 \in \Sym_{\bQ} (t^{-1}(L \otimes \bQ)[t^{-1}])$, such that $L \otimes \bQ$ is the degree 1 subspace and $Y(a(n),z) = \sum_{i \in \bZ} a(i) z^{-i-1}$ for all $a \in L \otimes \bQ$.  Furthermore each $\pi^L_{a,\bQ}$ is an irreducible $\pi^L_{0,\bQ}$-module.  Given a basis $\{e_j\}$ of $L$, with Gram matrix $G$ (so $G_{i,j} = (e_i,e_j)$), the vector $\sum_{i,j} \frac{(G^{-1})_{i,j}}{2} e_i(-1)e_j(-1)1$ is a conformal vector for $\pi^L_0$, giving it the structure of a vertex operator algebra of central charge equal to the rank of $L$.
\end{thm}

\begin{defn} \label{defn:twisted-group-ring} (\cite{FLM88} chapter 5)
Let $L$ be an even lattice, i.e., an integer lattice where $(a,a) \in 2\bZ$ for all $a \in L$.  We define the \textbf{double cover} $\hat{L}$ to be the unique (up to isomorphism) central extension of $L$ by $\pm 1$ such that the commutator $[a,b]$ is given by $(-1)^{(\bar{a},\bar{b})}$, where $\bar{a}, \bar{b} \in L$ are the images of $a, b\in \hat{L}$ under the covering homomorphism $\hat{L} \to L$.  We will write $e: L \to \hat{L}$, given by $\alpha \mapsto e_\alpha$, to denote a set-theoretic section of the covering homomorphism.  The \textbf{twisted group ring} $\bZ\{L\}$ of $L$ is the quotient of $\bZ[\hat{L}]$ by the ideal generated by $\kappa + 1$, where $\kappa$ is the nontrivial element in $\hat{L}$ that maps to $0 \in L$.  We write $\iota: \hat{L} \to \bZ\{L\}$, given by $a \mapsto a \otimes 1$ for the multiplicative map.  We note that there is an isomorphism $\bZ[L] \to \bZ\{L\}$ of abelian groups taking the basis vector $e^\alpha$ to $\iota(e_\alpha)$, but this is in general not a ring homomorphism.
\end{defn}

\begin{defn}
Let $L$ be an even lattice.  We define $V_{L,\bQ}$ to be the $\bQ\{L\}$-module $\bQ\{L\} \otimes_\bQ \pi^L_{0,\bQ}$, with $L$-graded decomposition $V_{L,\bQ} \cong \bigoplus_{a\in L} \pi^L_{a,\bQ}$ into Heisenberg modules.  For any $a \in L \otimes \bQ$, we define the operators
\[ \begin{aligned} E^-(-a,z) &= \exp\left( \sum_{n>0} \frac{a(-n)}{n}z^n \right)  \\
E^+(-a,z) &= \exp\left( \sum_{n>0} \frac{-a(n)}{n}z^{-n} \right) \end{aligned} \]
Furthermore, we define the operator $z^a$ to be scalar multiplication by $\langle a, b \rangle$ on $\pi^L_{b,\bQ}$, and $\iota(e_a)$ to be left multiplication by $\iota(e_a) \in \bQ\{L\}$.
\end{defn}

\begin{thm}\label{thm:lattice-VOA-Q} (\cite{B86}, \cite{FLM88})
Let $L$ be an even lattice.  Then, there is a unique $\bQ$-vertex algebra structure on $V_{L,\bQ}$ such that $1 \otimes e_0$ is the identity element, and $Y(e_a,z) = E^-(-a,z)E^+(a,z) z^a \iota(e_a)$ for all $a \in L$.  Fix a basis $\{\alpha_i\}$ of $L$, and let $G$ be the corresponding Gram matrix.  If $L$ is positive definite, then the vector $\omega = \sum_{i,j} \frac{(G^{-1})_{i,j}}{2} \alpha_i(-1)\alpha_j(-1)1$ endows $V_{L,\bQ}$ with the structure of a  vertex operator algebra over $\bQ$ with central charge equal to the rank of $L$.  Furthermore, $V_{L,\bQ}$ has an $L$-grading, in the sense that restriction of the product to $\pi^L_{a,\bQ} \otimes \pi^L_{b,\bQ}$ takes values in $\pi^L_{a+b,\bQ}((z))$.
\end{thm}

\begin{defn}
Let $L$ be an even lattice, and let $a \in L$.  We write $s_{a,n}$ to be the $z^n$-coefficient of $E^-(a,z)$, and $V_L$ for the $\bZ$-span of all composites $s_{a^1,n_1}\cdots s_{a^k,n_k}\iota(e_a)$ for $a^1,\ldots,a^k,a \in L$, $n_1,\ldots,n_k \in \bZ_{\geq 0}$.
\end{defn}

\begin{thm} (\cite{B86}, \cite{P92}, \cite{DG12})
Let $L$ be an even lattice.  Then, $V_L$ is a $\bZ$-graded vertex algebra over $\bZ$, and is a $\bZ$-form of $V_{L, \bQ}$.  If $L$ is positive definite and unimodular, then $V_L$ is a vertex operator algebra over $\bZ$ of central charge equal to the rank of $L$, and its graded pieces are positive definite unimodular lattices.
\end{thm}

\begin{defn}
For any commutative ring $R$, we write $V_{L,R}$ to denote $V_L \otimes R$, and for each $a \in L$, we define $\pi^L_{a,R}$ to be the $a$-graded part of $V_{L,R}$ in the lattice grading.  By the previous theorem, this notation does not create a conflict in the case $R = \bQ$.
\end{defn}

\begin{prop} \label{prop:lattice-VOA-subring-Q}
Let $L$ be an even lattice.  For any subring $R$ of $\bQ$, the vector $\omega \in V_{L,\bQ}$ lifts to an element in $V_{L,R}$ if and only if the determinant of $L$ is invertible in $R$.  If this is the case, and if $L$ is positive definite, then this lift of $\omega$ endows $V_{L,R}$ with a vertex operator algebra structure over $R$ of central half-charge equal to half the rank of $L$.
\end{prop}
\begin{proof}
The first claim for the special case $R = \bZ$ is Proposition 5.8 of \cite{M14}.  We give a sketch of how the argument extends to subrings of $\bQ$.

Given a basis $\{\alpha_1,\ldots,\alpha_l\}$ of $L$ and the dual basis $\{\alpha'_1,\ldots,\alpha'_l\}$ in $L \otimes \bQ$, one notes that the constants $c_{ij}$ defined by $\alpha'_i = \sum_{j=1}^l c_{ji}\alpha_i$ are the entries of the inverse Gram matrix $(G^{-1})_{i,j}$.
A short calculation (identical to \textit{loc. cit.}) shows that $\omega \in V_{L,R}$ if and only if $(G^{-1})_{i,i} \in 2R$ for all $i$ and $(G^{-1})_{i,j} + (G^{-1})_{j,i} \in 2R$ for all $j \neq i$.  By symmetry of the Gram matrix, the second condition is equivalent to $(G^{-1})_{i,j} \in R$ for all $j \neq i$.  If $\det L \in R^\times$, then $(G^{-1})_{i,j} \in R$ for all $i,j$, and evenness of $L$ implies $(G^{-1})_{i,i} \in 2R$.  Conversely, if $(G^{-1})_{i,j} \in R$ for all $i,j$, then $\alpha'_i \in L \otimes R$ for all $i$, so the determinant of $L$ is invertible in $R$.

For the second claim, it suffices to note that all of the defining properties of a vertex operator algebra are preserved and reflected by base change along $R \to \bQ$.  Thus, the claim follows from Theorem \ref{thm:lattice-VOA-Q}.
\end{proof}

\begin{cor}
Let $L$ be an even lattice, and let $R$ be a commutative ring.  If $\det L$ is invertible in $R$, then the formula for $\omega$ in Theorem \ref{thm:lattice-VOA-Q} defines a conformal vector $\omega$ that endows $V_{L,R}$ with the structure of a vertex operator algebra over $R$.
\end{cor}
\begin{proof}
Suppose $\det L$ is invertible in $R$, and let $S$ be the maximal subring of $\bQ$ that admits a ring homomorphism to $R$.  Then, $\det L$ is invertible in $S$, and by Proposition \ref{prop:lattice-VOA-subring-Q}, the formula for $\omega$ yields a vertex operator algebra structure on $V_{L,S}$.  This base-changes to a vertex operator algebra over $R$ by the unique ring homomorphism $S \to R$.
\end{proof}

Note that if we specialize this result to the case where $R$ is a field of characteristic $p > 0$, we obtain Theorem 14 in \cite{Mu14}.

\section{Reductive group schemes}

\subsection{Fundamentals}

We review the parts of the theory of reductive groups that are necessary for us. All definitions are taken from \cite{SGA3}. A basic modern reference is section 1 of \cite{C14}.

\begin{defn}
For any commutative ring $R$, the \textbf{multiplicative group} $\bG_{m,R}$ is the group scheme $\Spec R[x,y]/(xy-1) = \Spec R[x,x^{-1}]$ with multiplication rule given by the coproduct $\Delta(x) = x \otimes x$ and unit given by the counit $\epsilon(x) = 1$.  As a functor, $\bG_{m,R}$ takes a commutative $R$-algebra $S$ to the abelian group $S^\times$.  We define the \textbf{contravariant duality functor} $D_R$ from commutative group sheaves on $\Spec R$ to commutative group sheaves on $\Spec R$, by $D_R(M) = \uHom_{R-gp}(M, \bG_{m,R})$.  That is, for any commutative ring homomorphism $R \to S$, the group of $S$-points is $D_R(M)(S) = \Hom_S(M_S,\bG_{m,S})$.  A \textbf{torus} is a group scheme that represents $D_R(M)$ for $M$ an fpqc locally constant sheaf whose fibers are free abelian groups of finite rank.  If the sheaf $M$ is constant, then we say the torus $D_R(M)$ is split.
\end{defn}

Note: By \cite{SGA3} Exp. X Corollary 5.9, $D_R$ gives an involutive bijection between \'etale locally constant sheaves of free abelian groups of finite rank and tori over $R$.  In particular, (by \textit{loc. cit.} Corollary 4.5) any torus is trivialized by some \'etale cover.  Choosing a basis of a free abelian group $M$ of rank $n$ induces an isomorphism $D_R(M) \cong \bG_{m,R}^n$.

\begin{defn}
Given an affine group scheme $G$ over a commutative ring $R$, a \textbf{maximal torus} in $G$ is a subgroup scheme $T$ of $G$ that is a torus, such that for any geometric point $\Spec k \to \Spec R$, the torus $T_k$ is not contained in any strictly larger torus in $G_k$.
\end{defn}

By \cite{SGA3} Exp. XIV Corollary 3.20, maximal tori of smooth group $R$-schemes exist Zariski locally on $\Spec R$.  By \cite{SGA3} Exp. XXII Corollary 5.3.11, any two maximal tori (of an fpqc locally split reductive group) are conjugate \'etale locally.  This fact will play an essential role in our analysis of $\Aut V_L$.

\begin{defn} (\cite{SGA3} Exp. XIX Def. 1.6.1, 2.7)
An affine group scheme over a field is \textbf{reductive} if it is smooth, geometrically connected, and its maximal connected unipotent normal subgroup is trivial (here, a group is unipotent if it admits an increasing filtration whose quotients are subgroups of the additive group $\bG_a$, the group scheme which takes a ring to its underlying additive group).  Given a commutative ring $R$, a \textbf{reductive group scheme} over $R$ is a group scheme $G \to \Spec R$ whose structure morphism is smooth and affine, and whose geometric fibers are reductive in the previous sense.
\end{defn}

\begin{rem}
There are other definitions of ``reductive'' in the literature, where fibers are not necessarily connected (e.g., \cite{C14}). We follow the definition in \cite{SGA3} since that is where we draw the results we need.
\end{rem}

\begin{defn}
A \textbf{root datum} is a 4-tuple $(X,\Phi,X^\vee,\Phi^\vee)$, where
\begin{enumerate}
\item $X$ and $X^\vee$ are finite rank free abelian groups, equipped with a perfect duality pairing $\langle-,-\rangle: X \times X^\vee \to \bZ$
\item $\Phi \subset X$ and $\Phi^\vee \subset X^\vee$ are finite subsets, such that there exists a bijection $a \leftrightarrow a^\vee$ that satisfies $\langle a , a^\vee \rangle = 2$ and the associated reflections $s_a: x \mapsto x - \langle x,a^\vee \rangle a$ and $s_{a^\vee}: \lambda \mapsto \lambda - \langle a, \lambda \rangle a^\vee$ satisfy $s_a(\Phi) = \Phi$ and $s_{a^\vee}(\Phi^\vee) = \Phi^\vee$.
\end{enumerate}
A root datum is \textbf{reduced} if for any root $a \in \Phi$, the only scalar multiples that are roots are $a, -a$. A root datum is \textbf{semisimple} if the saturation of $\bZ \Phi$ in $X$ is all of $X$, or equivalently if $\bQ \Phi = X \otimes \bQ$.  
\end{defn}

\begin{thm}
Let $G$ be a reductive group over an algebraically closed field.  For a fixed maximal torus $T$, the 4-tuple $(X(T),\Phi(G,T),X^\vee(T),\Phi^\vee(G,T))$
defined as follows is a root datum:
\begin{enumerate}
\item $X(T) = \Hom_{k-gp}(T,\bG_m)$, the group of characters of $T$.
\item $X^\vee(T) = \Hom_{k-gp}(\bG_m, T)$, the group of cocharacters of $T$, dual to $X(T)$.
\item $\Phi(G,T)$ is the set of roots, i.e., nonzero characters of the $T$-action on $\fg = \Lie(G)$.  Each root $a$ induces a homomorphism (unique up to $T(k)$-conjugation on $G$) $\phi_a: SL_2 \to G$ taking the diagonal torus to $T$ and the strictly upper triangular matrices isomorphically to the root group $U_a$.
\item $\Phi^\vee(G,T)$ is the set of cocharacters induced by restricting $\phi_a$ to the diagonal torus of $SL_2$.
\end{enumerate}
\end{thm}

By \cite{SGA3} Exp. XIV Corollary 3.20, for any reductive group scheme $G$ over a commutative ring $R$, and any point $x \in \Spec R$, there is a Zariski open neighborhood $U$ of $x$ such that $G_U$ has a maximal $U$-torus.  Since such a torus is \'etale-locally split, we consider triples $(G,T,M)$, where $T = D_R(M)$ is split.  Then, the adjoint $T$-action on the Lie algebra $\fg = \Lie(G)$ endows the vector bundle $\fg$ with an $M$-grading, splitting into a direct sum of $\Lie(T)$ and line bundles $\fg_a$.  The corresponding sections $a \in M$ are called roots of the Lie algebra.

\begin{thm} \label{thm:existence-of-split-reductive} (\cite{SGA3} Exp. XXV 1.2) 
Given a reduced root datum $(X,\Phi,X^\vee,\Phi^\vee)$, there exists a pair $(G,T)$ where $G$ is a reductive group over $\bZ$ and $T$ is a split maximal torus, such that
\[ (X,\Phi,X^\vee,\Phi^\vee) \cong (X(T),\Phi(G,T),X^\vee(T),\Phi^\vee(G,T)). \]
In particular, $G$ is smooth and affine over $\bZ$ with connected fibers, and its Lie algebra has distinguished Cartan subalgebra $X$ and roots given by $\Phi$.
\end{thm}


\begin{defn}
A reductive group scheme $G \to \Spec R$ is \textbf{split} if it contains a maximal torus $T$ equipped with an isomorphism $T \simto D_R(M)$ for $M$ a finite rank free abelian group (i.e., $T$ is split), and the following conditions hold:
\begin{enumerate}
\item The nontrivial weights $a: T \to \bG_m$ that occur on $\fg$ arise from elements of $M$.
\item Each root space $\fg_a$ is free of rank 1 over $R$.
\item Each coroot $a^\vee: \bG_m \to T$ arises from an element of the dual lattice $M^\vee$.
\end{enumerate}
We write $(G,T,M)$ for a split triple.
\end{defn}

We note that a lot of information about a split reductive group can be easily read off the root datum.  For example, the center is $Z_G = D_R(M/Q)$, where $Q \subset M$ is the $\bZ$-span of roots.

\begin{lem} \label{lem:root-datum-attached-to-a-lattice}
Given an even positive definite lattice $L$, the 4-tuple $(X,\Phi,X^\vee,\Phi^\vee)$ defined as follows is a root datum:
\begin{enumerate}
\item $X = L$
\item $X^\vee = L^\vee$.  Since $L$ is equipped with an integral bilinear form, we have $L \subset L^\vee$.
\item $\Phi = \{ a \in L | \langle a, a \rangle = 2 \}$
\item $\Phi^\vee$ is $\Phi$, viewed as a subset of $L^\vee$.
\end{enumerate}
\end{lem}
\begin{proof}
This is easy to check.
\end{proof}

\begin{defn} \label{defn:GL}
Let $L$ be an even positive definite lattice.  We define the split reductive group $G_L$ over $\bZ$ as the triple $(G,T,L)$ attached to the root datum given in Lemma \ref{lem:root-datum-attached-to-a-lattice}.
\end{defn}

\begin{lem} \label{lem:basic-properties-of-GL} (basic properties of our group)
$G_L$ is a smooth affine group scheme over $\bZ$ with connected fibers. The Lie algebra of $G_L$ is the Lie algebra over $\bZ$ given by the root datum $(X,\Phi,X^\vee,\Phi^\vee)$, and in particular is isomorphic to the weight 1 subspace $(V_L)_1$, where the Cartan subalgebra $X$ is identified with the weight 1 subspace of $\pi^L_0$, with roots given by $\Phi$.  The semisimple group $D(G_L)$ is of ADE type (i.e., it is simply laced), and there is a central isogeny $Z \times D(G_L) \to G_L$ with finite central kernel.
\end{lem}
\begin{proof}
The scheme-theoretic claims and the identification of the Lie algebra are special cases of Theorem \ref{thm:existence-of-split-reductive}.  The simply-laced property follows from the fact that all roots have the same norm in the invariant inner product.  The central isogeny is a general fact about reductive groups, and is shown in \cite{SGA3} Exp XXII 6.2.4.
\end{proof}

\begin{defn}
Given a split reductive group $(G,T,M)$ over a nonempty scheme $S$ with root datum $R(G,T,M) = (M, \Phi, M^\vee,\Phi^\vee)$, a \textbf{pinning} is a pair $(\Phi^+, \{X_a\}_{a \in \Delta})$, where $\Phi^+$ is a ``positive'' system of roots (equivalently, a base of $\Phi$) and $X_a \in \fg_a(S)$ are trivializing sections for each simple positive root $a$. 
\end{defn}

We conclude this section with a result describing how an action of a split reductive group on a module is deduced from compatible actions of the torus and root groups.

\begin{thm} \label{thm:extend-action-to-G}
Let $(G,T,M,\Phi^+, \{X_a\})$ be a pinned split reductive group over a commutative subring $R$ of $\bC$.  Let $V$ be a finite rank free $R$-module, equipped with actions of $U_a$ (for all $a \in \Phi$), and $T$, together with an action of $G_\bC$ on $V_\bC$, satisfying the following conditions:
\begin{enumerate}
\item The two actions of $(U_a)_{\bC}$, given by restricting from $G_\bC$ and by taking base change of the action of $U_a$, coincide.
\item The two actions of $T_\bC$, given by restricting from $G_\bC$ and by taking base change of the action of $T$, coincide.
\end{enumerate}
Then, there is a unique action of $G$ on $V$ that restricts to the actions of $T$ and $U_a$, and whose base change to $\bC$ is the given action of $G_\bC$.  In other words, there is an equivalence of tensor categories between finite rank free $R$-modules $V$ with $G$-action, and finite rank free $R$-modules $V$ with actions of $T$ and all $U_a$ and an action of $G_\bC$ on $V_\bC$, such that base change to $\bC$ yields identical actions.
\end{thm}
\begin{proof}
By \cite{SGA3} Exp. XXIII Theorem 2.3 (a slightly different proof can be found in \cite{C14} Theorem 6.2.4), a collection of maps $f_T: T \to GL(V)$ and $f_a: U_a \to GL(V)$ (for all $a \in \Phi$) arise from the restriction of a map $f: G \to GL(V)$ (and this map is then unique) if and only if the following compatibilities hold:
\begin{enumerate}
\item $f_T(t)f_a(u_a)f_T(t)^{-1} = f_a(tu_at^{-1})$
\item $h_a f_T(t)h_a^{-1} = f_T(s_a(t))$
\item $h_a^2 = f_T(a^\vee(-1))$
\item $(h_a h_b)^{m_{ab}} = f_T(t_{ab})$
\item $(h_a f_a(\exp_a(X_a)))^3 = a$
\item $h_a f_{ab}(u_c)h_a^{-1} = f_{ab}(n_a u_c n_a^{-1})$
\item $h_b f_{ab}(u_c)h_b^{-1} = f_{ab}(n_b u_c n_b^{-1})$
\end{enumerate}
Here, for each root $a \in \Phi$, we have $X_a$ a distinguished trivializing section of $\fg_a$, the corresponding distinguished element $\exp_a(X_a)$ of the root group $U_a$, and a lift
\[ h_a = f_a(\exp_a(X_a))f_{-a}(\exp_{-a}(-X_a^{-1}))f_a(\exp_a(X_a)) \in GL(V) \]
of reflection in $a$ to $N_G(T)$.

Each compatibility condition asserts equality of two maps between affine group schemes over $R$, so by the reduced property of the source and separated property of the target, it suffices to check equality on a dense subset of the source.  Such a dense subset is given by the complex points, and the equality then holds because after base change to $\bC$, the maps come from restriction of a map from $G_\bC$.

We conclude that the actions of $T$ and $U_a$ on $V$ are restrictions of a unique action of $G$ on $V$.
\end{proof}

\subsection{The action of the reductive part}

\begin{prop} \label{prop:torus-action}
The torus $T = D(L)$ acts faithfully on $V_L$ by taking any $g \in D(L)(R) = \Hom(L, R^\times)$, for $R$ a commutative ring, to the automorphism of $V_{L,R}$ that acts on $\pi^L_{\lambda,R}$ by the invertible scalar $g(\lambda) \in R^\times$.
\end{prop}
\begin{proof}
The product on $V_L$ is uniquely determined by its restriction to Heisenberg submodules, and the restriction takes $\pi^L_{\lambda,R} \otimes \pi^L_{\mu,R}$ to $\pi^L_{\lambda + \mu,R}((z))$.  Thus, for any $g \in D(L)(R)$, $u \in \pi^L_{\lambda,R}$ and $v \in \pi^L_{\mu,R}$, we have
\[ \begin{aligned}
Y(g\cdot u, z) (g \cdot v) &= Y(g(\lambda)u,z)g(\mu)v \\
&= g(\lambda)g(\mu) Y(u,z)v\\
&= g(\lambda + \mu)Y(u,z)v \\
&= g\cdot Y(u,z)v
\end{aligned} \]
Thus, this is an action by weight-preserving vertex algebra automorphisms.  Faithfulness follows from the fact that $g \in D(L)(R)$ acts by identity if and only if $g(\lambda) = 1$ for all $\lambda \in L$, meaning $g$ is the trivial element in $\Hom(L, R^\times)$.
\end{proof}

\begin{prop}
For any root $\alpha \in L$ and any non-negative integer $n$, the $n$th-power of $\alpha_0$ is a multiple of $n!$ in $\End V_L$.  Furthermore, $\alpha_0^n$ is nonzero only for finitely many positive integers $n$.  In particular, there is a faithful action of the root group $U_\alpha \subset G_L$ on $(V_L)_\bZ$ by taking any $\exp(rX_\alpha) \in U_\alpha(R)$ to the locally finite sum $\sum_{n\geq 0} \frac{(r\alpha_0)^n}{n!}$.
\end{prop}
\begin{proof}
By Lemma 3.3 of \cite{B98}, the $n$th-power of $\alpha_0$ is $n!$ times an endomorphism of the vertex algebra $V_L$ over $\bZ$.  The formal sum defining $\exp \alpha_0$ is locally finite by degree considerations, since iterated translation by $\alpha$ eventually takes a weight $k$ vector to a Heisenberg module with lowest weight strictly greater than $k$ (an alternative way to say this is that positive-definiteness of $L$ implies only finitely many $\lambda \in L$ satisfy the property that $\pi^L_{\lambda,R}$ has lowest weight at most $k$).  Thus, $\exp(rX_\alpha)$ for $r$ ranging over elements of commutative rings $R$ give compatible automorphisms of $(V_L)_R$, generating an action of $\bG_a$ on $(V_L)_\bZ$ and $\bG_{a,R}$ on $(V_L)_R$.
\end{proof}

\begin{prop}
The actions of the torus $D(L)_\bC$ and the root groups $U_{\alpha,\bC}$ on $(V_L)_\bC$ coincide with the actions of the corresponding subgroups of the complex Lie group $N = \langle \exp(h_0) | h \in (V_L)_{\bC,1} \rangle$.
\end{prop}
\begin{proof}
We note that the Lie algebras $N$ and $(G_L)_\bC$ are both identified with $(V_{L,\bC})_1$.  Furthermore, the groups $N$, $D(L)_\bC$, and $(U_\alpha)_\bC$ all act faithfully on $V_{L,\bC}$, so it suffices to show that the two infinitesimal actions of the torus and root groups coincide.

To show the torus actions coincide, we note that the first-order infinitesimal action of $L \otimes \bC \subset (V_{L,\bC})_1$ is by $h \mapsto 1 + \epsilon h_0$, which takes any $v \in \pi^L_{\lambda,\bC}$ to $v + \epsilon \langle h, \lambda \rangle v$.  The first-order infinitesimal action of the torus $D(L)_\bC$ is the map that takes, for $\nu \in \Hom(L,\bC)$, the element $1 + \epsilon \nu \in \Hom(L,\bC[\epsilon]/(\epsilon^2)^\times) = D(L)(\bC[\epsilon]/(\epsilon^2))$ to the automorphism that takes any $v \in \pi^L_{\lambda,\bC}$ to $(1 + \epsilon \nu)(\lambda) v = v + \epsilon \nu(\lambda) v$.  These match when we identify $\Hom(L,\bC)$ with $L \otimes \bC$ via the inner product.

For the root groups, the first-order infinitesimal action of $(U_\alpha)_\bC \subset (G_L)_\bC$ is given by $h \mapsto 1 + \epsilon h_0$, and the same is true for the corresponding subgroup of $N$.
\end{proof}

\begin{thm}
There is a unique action of $G_L$ on $V_L$ that restricts to the given actions of $T$ and $U_a$, such that base change to $\bC$ yields the action of $(G_L)_\bC$ on $(V_L)_\bC$ given by \cite{DN98} (where $(G_L)_\bC$ is given the notation $N$).
\end{thm}
\begin{proof}
From Theorem \ref{thm:extend-action-to-G}, there is an equivalence of tensor categories between:
\begin{enumerate}
\item finite rank free $\bZ$-modules $V$ with $G_L$-action, and
\item finite rank free modules $V$ with actions of $D(L)$ and $U_\alpha$ together with an action of $(G_L)_\bC$ on $V_\bC$ such that, after base change to $\bC$, the actions of the subgroups coincide with the restrictions of the $(G_L)_\bC$ action.
\end{enumerate}
For finite sums of homogeneous subspaces of $V_L$, these data are given by the previous three propositions.  Each homogeneous component of $V_L$ therefore admits a unique action of $G_L$ compatible with the actions of $T$, $U_a$, and $(G_L)_\bC$. Furthermore, each coefficient of each homogeneous vertex operator, as a map from a tensor product of homogeneous spaces to another homogeneous space, is equivariant with respect to this action.  We conclude that the actions of $G_L$ on the homogeneous subspaces are restrictions of a uniquely defined action on $V_L$ by homogeneous vertex algebra automorphisms.
\end{proof}

\section{Symmetries of the Mewtwo cover}

In section \ref{sec:lattice}, we defined a double cover $\hat{L}$ of an even lattice $L$ as an abstract group, and used it to construct the Fock space $V_{L,\bQ}$.  There is an embedding from the orthogonal symmetry group of $\hat{L}$ to the symmetry group of $V_{L,\bQ}$ and also to that of $V_L$, but it acquires a kernel if we base change to characteristic 2.  In order to have a faithful action when 2 is not invertible, we will replace the abstract double cover $\hat{L}$ with a finite flat group scheme $\tilde{L}$, presented as a $\mu_2$-central extension of the constant group scheme $L$.  For any ring $R$ where $\mu_2(R) = \{1, -1\}$ with $1 \neq -1$, there is a canonical identification $\tilde{L}(R) \cong \hat{L}$.  This condition holds for any integral domain of characteristic not 2, so our replacement is not a particularly drastic change.

The symmetry group of this cover is then an extension $O(\tilde{L})$ of $O(L)$ by $\Hom(L, \mu_2) \cong \mu_2^{\rank L}$.  We show that $O(\tilde{L})$ acts faithfully on $V_L$.

\subsection{The cover of a lattice}

We once again fix a base ring $R$.  Recall that $\mu_{2,R}$ is the group scheme that assigns to each commutative $R$-algebra $S$ the set $\mu_{2,R}(S) = \{x \in S | x^2 = 1\}$, with multiplication induced by multiplication in $S$.  It is represented by the finite flat scheme $\Spec R[x]/(x^2-1)$ with comultiplication $\Delta(x) = x \otimes x$.  This scheme is non-reduced over 2, but is smooth away from 2.  If $2 \neq 0$ in $R$, then the underlying affine scheme has a cyclic automorphism group of order 2 induced by $x \mapsto -x$, and automorphism group scheme $\mu_2$, while the group scheme itself has trivial automorphism group.  

\begin{defn}
Let $A$ be an abelian group, and $q: A \to \bZ$ a quadratic form with corresponding even bilinear form $B_q(x,y) = q(x + y) - q(x) - q(y)$.  A \textbf{$\mu_2$-cover} of $(A,q)$ over $R$ is a group scheme $\tilde{A}$ over $R$ equipped with a surjective homomorphism $\pi: \tilde{A} \to \uA_R$ to the constant group scheme $\uA_R = \coprod_{a \in A} \Spec R$, satisfying the following properties:
\begin{enumerate}
\item As a scheme, $\tilde{A} \cong \coprod_{a \in A} \mu_2$, and $\pi$ sends each copy of $\mu_2$ along its structure morphism to the corresponding copy of $\Spec R$.
\item For any commutative $R$-algebra $S$, the kernel of $\pi_S: \tilde{A}(S) \to \uA(S)$ is isomorphic to $\mu_2(S) = \{ x \in S | x^2 = 1\}$, and this subgroup lies in the center of $\tilde{A}(S)$.
\item For each commutative $R$-algebra $S$, and each $a,b \in \tilde{A}(S)$, we have $ab = (-1)^{B_q(\pi(a),\pi(b))}ba$, where $B_q: \uA(S) \times \uA(S) \to \underline{\bZ}(S)$ is the bilinear form extended to a map of sheaves.
\end{enumerate}
Given $\mu_2$-covers $\pi: \tilde{A} \to \uA$ and $\pi': \tilde{A}' \to \uA$, an isomorphism of $\mu_2$-covers is a group scheme isomorphism $\phi: \tilde{A} \to \tilde{A}'$ that satisfies $\pi = \pi' \circ \phi$.
\end{defn}

\begin{lem} \label{lem:existence-and-uniqueness-of-covers}
Let $A$ be a free abelian group of finite rank, and $q: A \to \bZ$ a quadratic form.  Then there exists a $\mu_2$-cover $\pi: \tilde{A} \to A$ of $(A,q)$.  Furthermore, this $\mu_2$-cover is unique up to isomorphism of $\mu_2$-covers.  The center of $\tilde{A}$ contains $\pi^{-1}(A^\perp)$ and $\pi^{-1}(2A)$.
\end{lem}
\begin{proof}
We first define $\tilde{A} = \uA \times \mu_2$ as a scheme.  We then prove existence of a suitable group operation by choosing a cocycle over $\bZ$ and base-changing it to all commutative rings $R$.  Following Proposition 5.2.3 of \cite{FLM88}, we let $c$ be the alternating bimultiplicative map from $A \times A \to \mu_2(\bZ) = \{ \pm 1\}$ given by $c(a,b) = (-1)^{\langle a,b \rangle}$, and define a cocycle $\epsilon: A \times A \to \mu_2(\bZ)$ by choosing an ordered basis $\alpha_1,\ldots,\alpha_n$, and taking the unique bimultiplicative map such that $\epsilon(\alpha_i,\alpha_j) = \begin{cases} c(\alpha_i,\alpha_j) & i > j \\ 1 & i \leq j \end{cases}$.  For any commutative $R$-algebra $S$, we then define a cocycle $\tilde{\epsilon}: \uA(S) \times \uA(S) \to \mu_2(S)$ by noting that $\uA(S)$ is the group of locally constant functions $\Spec S \to A$, so on the open sets $U$ where a pair of points is constant, we define the restriction of $\tilde{\epsilon}$ to $U$ as composition of a pair of such functions with $\epsilon$, followed by the canonical map $\mu_2(\bZ) \to \mu_2(S)$.  This cocycle then defines a functorial central extension.

For uniqueness, we introduce some notation to describe the structure of $\mu_2$-covers.  Let $\pi: \tilde{A} \to \uA$ and $\pi': \tilde{A}' \to \uA$ be $\mu_2$-covers of $(A,q)$.  We wish to construct an isomorphism of $\mu_2$-covers between them.  By Proposition 5.2.3 in \cite{FLM88}, there is a unique double cover $f: \hat{A} \to A$ with commutator map $ab = (-1)^{B_q(f(a),f(b))}ba$, up to isomorphism.  Thus, there exists an isomorphism $\tilde{A}(S) \to \tilde{A}'(S)$ that induces equality on $\uA(S)$ and the kernels $\mu_2(S)$.  It remains to show that this group isomorphism can be promoted to an isomorphism of group schemes.  

We view $\tilde{A}$ and $\tilde{A}'$ as disjoint unions of copies of $X = \Spec R[x]/(x^2-1)$, indexed by elements of $A$.  For each $\bar{a},\bar{b} \in A$, the multiplication operations restricted to the $\bar{a},\bar{b}$ components are maps from copies of $X \times X$ to $X$.  Maps between copies of $X$ and its finite products are uniquely determined by the corresponding maps on $R$-points, so any condition that can be expressed as commutativity of a diagram of copies of $X$ can be checked on $R$-points.  In particular, the existence of a group scheme isomorphism between group schemes that are disjoint unions of copies of $X$ is a condition that satisfies this property.  Thus, the existence of a $\mu_2$-cover isomorphism $\tilde{A} \to \tilde{A}'$ reduces to the corresponding isomorphism on $R$-points, which we have established.

For the claim about the center, we note that fixing any input to lie in $A^\perp$ or $2A$ yields trivial commutator maps, so the $\bZ$-points of $\pi^{-1}(A^\perp)$ and $\pi^{-1}(2A)$ are central.  By the same reasoning as the previous paragraph, the corresponding scheme-theoretic commutator maps are trivial.
\end{proof}

\subsection{Orthogonal transformations}

\begin{defn}
Let $(\tilde{A}, \pi)$ be the $\mu_2$-cover of a free abelian group $A$ with quadratic form $q: A \to \bZ$.  An \textbf{orthogonal transformation} of $(\tilde{A}, \pi)$ over $S$ is a group scheme automorphism $\phi$ of $\tilde{A}_S$ such that there exists an automorphism $\bar{\phi}$ of $A$ that preserves $q$, satisfying $\pi \circ \phi = \bar{\phi} \circ \pi$.  The functor that takes an $R$-algebra $S$ to the group of orthogonal transformations of $(\tilde{A},\pi)$ over $S$ is written $O(\tilde{A})$, or $O(\tilde{A},\pi)$.
\end{defn}

\begin{prop} \label{prop:structure-of-orthogonal-transformation-group}
Let $(A,q)$ be a free abelian group of finite rank $r$, with quadratic form $q: A \to \bZ$.  Let $O(A,q)$ be the constant group of orthogonal transformations of $(A,q)$.  Then, the sequence:
\[ 1 \to \uHom(\uA, \mu_2) \to O(\tilde{A},\pi) \to O(A,q) \to 1 \]
of group sheaves is exact.  That is, the automorphism group sheaf $O(\tilde{A},\pi)$ of the $\mu_2$-cover $(\tilde{A},\pi)$ of $\uA$ has a normal subgroup isomorphic to $\mu_2^r$ with quotient isomorphic to $O(A,q)$.  Here, the map $\uHom(\uA,\mu_2) \to O(\tilde{A},\pi)$ takes any $S$-homomorphism $\lambda: \uA_S \to \mu_{2,S}$ to $\lambda^*: \tilde{A}(S) \to \tilde{A}(S)$ defined by $a \mapsto (-1)^{\lambda(\bar{a})}a$.  In particular, $O(\tilde{A},\pi)$ is representable by a flat group scheme locally of finite presentation, and away from the prime 2 it is \'etale.  Furthermore, if $(A,q)$ is positive definite, then $O(\tilde{A},\pi)$ is finite, flat, and affine.
\end{prop}
\begin{proof}
For any orthogonal transformation $\phi \in O(\tilde{A},\pi)(S)$, the corresponding automorphism $\bar{\phi} \in O(A,q)(S)$ is uniquely determined by the condition $\pi \circ \phi = \bar{\phi} \circ \pi$.  The commutativity of the corresponding diagram for composites implies we have a canonical homomorphism $O(\tilde{A},\pi) \to O(A,q)$.  We will first show that it is surjective, then determine the kernel.

For any $\bar{\phi} \in O(A,q)(S)$, we may partition $\Spec S$ into finitely many open and closed subschemes $U$ such that restriction to $U$ is constant.  To produce an element $\phi \in O(\tilde{A},\pi)(S)$ mapping to $\bar{\phi}$, it suffices to produce elements $O(\tilde{A},\pi)(U)$ mapping to the corresponding restrictions.  Thus, it suffices to show that for any $h \in O(A,q)$, there is an orthogonal transformation of $\tilde{A}$ that maps to it.  The homomorphism $h \circ \pi: \tilde{A} \to A$ is a $\mu_2$-cover of $A$, so Lemma \ref{lem:existence-and-uniqueness-of-covers} implies there exists an automorphism $\tilde{h}$ of $\tilde{A}$ such that the diagram
\[ \xymatrix{ 1 \ar[r] & \mu_2 \ar@{=}[d] \ar[r] & \tilde{A} \ar[r]^{h \circ \pi} \ar[d]^{\tilde{h}} & A \ar@{=}[d] \ar[r] & 1 \\ 1 \ar[r] & \mu_2 \ar[r] & \tilde{A} \ar[r]_{\pi} & A \ar[r] & 1 } \]
commutes.  Then, $\tilde{h} \in O(\tilde{A},\pi)$, and maps to $h$ under our homomorphism.  This proves surjectivity.

We now consider the kernel.  We first note that for any $\lambda: \uA_S \to \mu_{2,S}$, the map $\lambda^*: \tilde{A}(S) \to \tilde{A}(S)$ is nontrivial if $\lambda$ is nontrivial, and induces the trivial map $\Spec S \to O(A,q)$, so $\uHom(\uA,\mu_2)$ lies in the kernel of the homomorphism.  If $g \in O(\tilde{A},\pi)(S)$ induces the trivial map $\Spec S \to O(A,q)$, then $\pi \circ g = \pi$ as $S$-group scheme homomorphisms $\tilde{A}_S \to \uA_S$.  For any $x \in A$, we have a subscheme $\hat{x} \in \uA_S$ corresponding to $x$ that is isomorphic to $\Spec S$, and $\pi^{-1}(\hat{x})$ is a $\mu_{2,S}$-torsor.  The restriction of $g$ to $\pi^{-1}(\hat{x})$ is then multiplication by some element of $\mu_{2,S}(S)$.  The group law on $\tilde{A}_S$ implies $g$ is uniquely determined by the restriction of $g$ to $\pi^{-1}(\{\hat{e}_1,\ldots,\hat{e}_r\})$, where $e_1,\ldots,e_r$ is a basis of $A$, and any such choice of $r$ elements of $\mu_{2,S}(S)$ corresponds to a unique element of $\uHom(\uA,\mu_2)(S)$.

We conclude that the sheaf of sets underlying $O(\tilde{A},\pi)$ is isomorphic to the direct product of $\mu_2^r$ and the constant group scheme $O(A,q)$, hence representable by a scheme.  If $q$ is positive definite, then $O(A,q)$ is a finite constant group, so the claimed properties of $O(\tilde{A},\pi)$ hold.
\end{proof}

\begin{prop}
 If $L$ is positive-definite, then $O(\tilde{L})$ is an affine group scheme of finite type.
\end{prop}
\begin{proof}
Given a basis of $L$, let $d$ be the maximal norm of vectors in this basis, and let $n$ be the number of vectors in $L$ of length at most $d$.  Then $O(\tilde{L})$ is a subgroup of $GL_n$, since it acts faithfully on the linear space whose basis is the set of these vectors.
\end{proof}

\subsection{Action of orthogonal transformations on the lattice vertex algebra}

Recall the twisted group ring $\bZ\{L\} = \bZ[\hat{L}]/(\kappa + 1)$, introduced in Definition \ref{defn:twisted-group-ring}.  For any commutative ring $R$, we write $R\{L\}$ for $R \otimes \bZ\{L\}$, and note the identification with $R[\hat{L}]/(\kappa + 1)$.

\begin{lem}
Let $L$ be an even lattice and let $R$ be a commutative ring.  There is an embedding of the $\mu_2(R)$-cover $\tilde{L}(R)$ into the units of the twisted group ring $R\{L\}$, with image $\{ r \iota(e_a)| a \in L, r \in \mu_2(R)\}$.
\end{lem}
\begin{proof}
Choose a basis $\{v_i\}$ of $L$.  For each basis vector $v_i$, the preimage of $v_i$ in $\tilde{L}(R)$ is a $\mu_2(R)$-torsor, and so is $\{ r \iota(e_{v_i}) | r \in \mu_2(R)\}$.  For each basis vector, we choose a $\mu_2(R)$-equivariant bijection between the two sets.  We claim that this collection of bijections extends uniquely to an isomorphism from $\tilde{L}(R)$ to the subgroup $\{ r \iota(e_a)| a \in L, r \in \mu_2(R)\}$ of $R\{L\}^\times$.

Since the preimages of basis vectors $v_i$ generate $\tilde{L}$, and $\{ r \iota(e_{v_i}) | r \in \mu_2(R)\}$ generate the subgroup of $R\{L\}^\times$ under consideration, we have uniqueness once we have existence.  Existence follows from the fact that both groups admit a surjective homomorphisms to $L$, and satisfy the same commutation relations.
\end{proof}

\begin{lem} \label{lem:faithful-action-on-twisted-group-ring}
Let $L$ be a positive-definite even lattice and let $R$ be a commutative ring.  The canonical action of $O(\tilde{L})(R)$ on $\tilde{L}$ extends to a faithful action on the twisted group ring $R\{L\}$ by $R$-algebra automorphisms.
\end{lem}
\begin{proof}
For any basis $\{v_i\}$ of $L$, the action of $g \in O(\tilde{L})(R)$ sends the generating set $\{r \iota(e_{v_i}) | r \in \mu_2(R)\}$ to the generating set $\{r \iota(e_{g \cdot v_i}) | r \in \mu_2(R)\}$, and extends to a group homomorphism on the subgroup of units generated by these sets, and since this commutes with scalar multiplication, we get an action by automorphisms on the multiplicative monoid of $R\{L\}$.  Finally, additive structure is preserved, since the action takes the rank one free modules $R\iota(e_a)$ to each other, applying only scalar multiplications.
\end{proof}

Recall that for any commutative ring $R$, $V_{L,R}$ is the $R$-span of all composites $s_{a^1,n_1}\cdots s_{a^k,n_k}\iota(e_a)$ for $a^1,\ldots,a^k,a \in L$, $n_1,\ldots,n_k \in \bZ_{\geq 0}$.

\begin{prop} \label{prop:OL-action-on-VOA}
Let $L$ be a positive-definite even lattice and let $R$ be a commutative ring. We define the map $O(\tilde{L})(R) \times V_{L,R} \to V_{L,R}$ by linearly extending the map $(g, s_{a^1,n_1}\cdots s_{a^k,n_k}\iota(e_a)) \mapsto s_{ga^1,n_1} \cdots s_{ga^k,n_k}g\iota(e_a)$ on basis vectors.  This describes a faithful group scheme action by vertex operator algebra automorphisms.
\end{prop}
\begin{proof}
The fact that this is a faithful group action follows from \ref{lem:faithful-action-on-twisted-group-ring}.  The fact that the action is by vertex operator algebra automorphisms follows from the multiplication formulas for elements of $V_{L,R}$ - in the case $L$ is the Leech lattice and $R$ is a $\bQ$-algebra, this is \cite{FLM88} Theorem 10.4.6, but the proof does not depend on the particular lattice.  When $R = \bZ$, the formulas follow from base change to $\bQ$, and the proof for general $R$ follows by base change from $\bZ$.
\end{proof}

\begin{lem} \label{lem:OL-normalizes-torus}
The action of $O(\tilde{L})$ on $V_L$ normalizes the $T_L$-action.
\end{lem}
\begin{proof}
This follows from the fact that for any $g \in O(\tilde{L})(R)$ and $\lambda \in L$, the corresponding automorphism of $V_{L,R}$ takes the Heisenberg module $\pi^L_{\lambda,R}$ to $\pi^L_{g\cdot \lambda,R}$.
\end{proof}

\section{The automorphism group scheme}

We show that $\uAut V_L$ is an affine group scheme, and that the two closed immersions from $O(\tilde{L},\pi)$ and $G_L$ describe a surjection $O(\tilde{L},\pi) \times G_L \to \uAut V_L$ by multiplication.

\subsection{Basic properties}

\begin{prop} \label{prop:aut-represented-by-affine-finite-type}
Let $L$ be a positive definite even lattice, and let $R$ be a commutative ring.  Then, the functor that assigns to each commutative ring $S$ over $R$, the group of homogeneous $S$-vertex algebra automorphisms of $V_{L,S}$, and to each ring homomorphism the corresponding group homomorphism, is represented by an affine group scheme of finite type over $\Spec R$.
\end{prop}
\begin{proof}
By Lemma \ref{lem:represented-by-affine}, it suffices to show that $V_{L,R}$ is finitely generated.  By Theorem 3.3 of \cite{DG12}, $V_{L,\bZ}$ is generated by the finite set $\{ e^{\pm \gamma_i}\}$ for $\gamma_i$ ranging over a basis of $L$.  Thus, the finite set $\{ e^{\pm \gamma_i} \otimes 1\}$ generates $V_{L,R}$.
\end{proof}

\begin{prop} \label{prop:normality}
Let $L$ be a positive definite even lattice, and let $R$ be a commutative ring.  Then the action of $\Aut V_{L,R}$ on the weight 1 subspace induces an automorphism of $(G_L)_R$.  That is, $(G_L)_R$ is a normal subgroup of $\Aut V_{L,R}$.
\end{prop}
\begin{proof}
In Proposition \ref{lem:basic-properties-of-GL}, we identified the weight 1 subspace of $V_{L,R}$ with the Lie algebra of $(G_L)_R$, in particular identifying the root spaces with the Lie subalgebras attached to the root groups.  Since any automorphism of $V_{L,R}$ fixes the weight 1 subspace, conjugation takes $(G_L)_R$ to a subgroup of $\Aut V_{L,R}$ with the same Lie algebra.  However, $(G_L)_R$ has connected fibers, so it is stabilized.
\end{proof}

\subsection{Local Conjugation}

We wish to show that $\Aut V_L$ is a product of $G_L$ and $O(\tilde{L})$.  The following lemma lets us reduce this problem to more local questions.

\begin{lem}
To show that $\Aut V_L$ is a product of $G_L$ and $O(\tilde{L})$, it suffices to show that, given a commutative ring $S$ and an $S$-point $g$ of $\Aut V_L$, there is a faithfully flat cover $\{U_i \to \Spec S\}$ such that each $g|_{U_i}$ lies in the image of the product maps on $U_i$.  
\end{lem}
\begin{proof}
By Lemma \ref{lem:faithfully-flat-sheaf}, $\Aut V_L$ is an faithfully flat sheaf, so to identify $\Aut V_L$ with the product, it suffices to show that the product map gives a sheaf surjection in the sense of \cite{Stacks} Definition 00WM.  This is precisely the condition given.
\end{proof}

As it happens, we don't need the full power of faithfully flat maps, and will only consider \'etale covers.

\begin{prop} \label{prop:conjugate-tori}
For any commutative ring $S$ and an $S$-point $g$ of $\Aut V_L$, there is an \'etale cover of $\{ S \to S_i\}$ and a choice of elements $g_i \in G_L(S_i)$ such that for all $i$, the conjugation action of $g_i (g|_{S_i})$ stabilizes $T_{L,S_i}$.
\end{prop}
\begin{proof}
By Proposition \ref{prop:normality}, $G_L$ is normal in $\Aut V_L$, so conjugation by $g$ takes $T_{L,S}$ to some torus $T'$ in $G_{L,S}$.  By \cite{SGA3} Exp. XXII Corollary 5.3.11, any two maximal tori (of an fpqc locally split reductive group) are conjugate \'etale locally on $S$.  Thus, there exists a choice of \'etale cover $\{ S \to S_i\}$ and $\{g_i \in G_L(S_i) \}$ so that conjugation by each $g_i$ takes $T'_{S_i}$ to $T_{L,S_i}$.
\end{proof}

\begin{cor} \label{cor:product-with-normalizer}
$\Aut V_L$ is a product of $G_L$ and $N_{\Aut V_L} T_L$.
\end{cor}
\begin{proof}
By Proposition \ref{prop:conjugate-tori}, the map $G_L \times N_{\Aut V_L} T_L \to \Aut V_L$ induced by multiplication is a surjection of \'etale sheaves.  Hence, it is a surjection of affine schemes.
\end{proof}

\subsection{Normalizer of the distinguished torus}

\begin{defn}
Consider the homomorphism $N_{\Aut V_L}T_L \to \Aut T_L$ induced by conjugation.  We let $c : N_{\Aut V_L}T_L \to GL(L)$ be the composite of this homomorphism with the isomorphism $\Aut T_L \simto GL(L)$.
\end{defn}

\begin{lem} \label{lem:normalizer-preserves-modules}
The action of $N_{\Aut V_L}T_L$ takes Heisenberg submodules of $V_L$ to Heisenberg submodules.  That is, for any commutative ring $R$, $\lambda \in L$, and $\sigma \in N_{\Aut V_L}T_L(R)$, there is some $\mu \in L$ such that for any $u \in \pi^L_{\lambda,R}$, we have $\sigma u \in \pi^L_{\mu,R}$.
\end{lem}
\begin{proof}
The orbit of $T_L$ acting on any nonzero vector in $\pi^L_{\lambda,R}$ is one dimensional, and contained in a line.  Let $u \in \pi^L_{\lambda,R}$, and write $\sigma u = \sum_{i \in I} u_i$ for $u_i \in \pi^L_{\mu_i,R}$, where the map $i \mapsto \mu_i$ is injective.  If the indexing set $I$ has more than one element, then the orbit of $T_L$ on $\sigma u$ is not contained in a line.  This contradicts the fact that $\sigma$ acts linearly and bijectively.  We conclude that $I$ is a singleton.
\end{proof}

\begin{lem} \label{lem:normalizer-action-compatible}
Points of $N_{\Aut V_L}T_L$ take Heisenberg submodules of $V_L$ to Heisenberg submodules in a manner compatible with the homomorphism $c$ to $GL(L)$.  Specifically, for any commutative ring $R$, $\lambda \in L$, and $\sigma \in N_{\Aut V_L}T_L(R)$, we have $\sigma (u) \in \pi^L_{c(\sigma)\cdot \lambda,R}$ for all $u \in \pi^L_{\lambda,R}$. 
\end{lem}
\begin{proof}
We may assume, by suitably replacing $R$ with a suitable $R$-algebra such as the group ring $R[L]$, that $R^\times$ admits an injective homomorphism $g$ from $L$.  By Lemma \ref{lem:normalizer-preserves-modules}, we see that $\sigma$ induces a linear isomorphism $\pi^L_{\lambda,R} \to \pi^L_{\mu,R}$ for some $\mu$.  Thus, for any $u \in \pi^L_{\lambda,R}$, we have
\[ \begin{aligned}
(\sigma^{-1}g\sigma)(\lambda)u &= \sigma^{-1} g \sigma u \\
&= \sigma^{-1} g(\mu) \sigma (u) \\
&= g(\mu) u,
\end{aligned} \]
where we freely switch between viewing elements of $T_L(R)$ as automorphisms of $V_{L,R}$ and as homomorphisms $L \to R^\times$.  Since $\sigma^{-1} g \sigma$ is also injective as a map $L \to R^\times$, this implies $c(\sigma)\lambda = \mu$.
\end{proof}

\begin{lem} \label{lem:image-of-normalizer}
The image of $c$ lies in the subgroup $O(L)$ of $GL(L)$, i.e., norms are preserved.
\end{lem}
\begin{proof}
By Lemma \ref{lem:normalizer-action-compatible}, the action on Heisenberg modules is compatible with the map to $GL(L)$, and so the lowest weights of Heisenberg modules are preserved. The result follows from the fact that the lowest weight is half of the norm of a lattice vector.
\end{proof}

\begin{prop} \label{prop:normalizer-as-product}
The normalizer of $T_L$ in $\Aut V_L$ is the product $O(\tilde{L})C_{\Aut V_L} T_L$.  In particular, $c$ induces an isomorphism $(N_{\Aut V_L}T_L)/C_{\Aut V_L}T_L \simto O(L)$ taking a normalizing automorphism to the induced orthogonal transformation on the lattice parametrizing Heisenberg modules in $V_L$.
\end{prop}
\begin{proof}
By Proposition \ref{prop:OL-action-on-VOA}, we know that $N_{\Aut V_L}T_L$ contains $O(\tilde{L})$.  Thus, it suffices to show $N_{\Aut V_L}T_L \subseteq O(\tilde{L})C_{\Aut V_L} T_L$.

The conjugation homomorphism $c: N_{\Aut V_L}T_L \to GL(L)$ has kernel $C_{\Aut V_L} T_L$ and by Lemma \ref{lem:image-of-normalizer}, its image lies in $O(L)$.  The centralizer of $T_L$ in $O(\tilde{L})$ is the kernel of the homomorphism to $O(L)$, and is therefore identified with $\Hom(L,\mu_2)$.  Thus, the image of $O(\tilde{L})$ under $c$ is precisely $O(L)$, and we have an isomorphism.
\end{proof}

We can now show that $T_L$ is its own centralizer in the automorphism group scheme of $V_L$.

\begin{prop} \label{prop:centralizer}
$C_{\Aut V_L} T_L = T_L$.
\end{prop}
\begin{proof}
Let $g \in C_{\Aut V_L} T_L (R)$ for some $R$.  Then, because $g$ preserves the $L$-grading of $V_L$, $g$ fixes the Heisenberg subalgebra $\pi^L_{0,R}$ of $V_{L,R}$ pointwise, and by Lemma \ref{lem:normalizer-action-compatible}, $g$ also fixes each Heisenberg module $\pi^L_{\lambda,R}$ setwise.  Since each Heisenberg module is generated by its free rank one lowest-weight space, $g$ acts on these spaces as invertible scalars.  The fact that $g$ is compatible with the vertex algebra structure on $V_{L,R}$ then implies these scalars are given by a group homomorphism $L \to R^\times$, so $g \in T_L(R)$.
\end{proof}

\begin{cor} \label{cor:homogeneous-aut-VOA-aut}
Let $L$ be an even lattice, and let $R$ be a commutative ring in which $\det L$ is invertible.  Then, the inclusion of the group of $R$-vertex operator algebra automorphisms of $(V_{L,R}, \omega)$ into the group of homogeneous $R$-vertex algebra automorphisms of $V_{L,R}$ is an isomorphism.
\end{cor}
\begin{proof}
It suffices to show that any homogeneous $R$-vertex algebra automorphism $\sigma$ of $V_{L,R}$ fixes $\omega$.  By Corollary \ref{cor:product-with-normalizer} and Proposition \ref{prop:normalizer-as-product}, we may write $\sigma$ as a product $gh$, where $g \in G_L(R')$ and $h \in O(\tilde{L})(R')$ for $R \hookrightarrow R'$ \'{e}tale.  The formula for $\omega$ can be written as a sum over dual basis vectors, so it is invariant under orthogonal transformations, and in particular, the action of $O(\tilde{L})$.  It therefore suffices to show that $\omega$ is $G_L$-invariant, but for this it suffices to show that the action of $T_L$ and the infinitesimal actions of root vectors $h_0$ are trivial.  Because $\omega$ lies in lattice-degree zero, $T_L$ acts trivially.  For any root vector $h$, the well-known commutator formula $[h_m,L(n)] = mh_{m+n}$ (\cite{FLM88} eq. 8.7.20) implies $h_0 \omega = h_0 L(-2) \unit = L(-2) h_0 \unit = 0$.
\end{proof}

We now consider the intersection of $O(\tilde{L})$ and $G_L$ in $\Aut V_L$.

\begin{defn}
Let $L$ be an even lattice.
\begin{itemize}
\item The \textbf{Weyl group} $W_L$ of $L$ is the subgroup of $O(L)$ generated by reflections in norm 2 vectors.
\item The \textbf{outer automorphism group} of $L$ is the quotient $O(L)/W_L$.
\item The \textbf{Tits group} $\tilde{W}_L$ of $G_L$ is the subgroup of $N_{G_L}T_L$ generated by $N_{G_L}(T_L)(\bZ)$-translates of the 2-torsion subgroup of the torus $\Hom(L,\mu_2) = T_L[2] \subset T_L$ (see example 6.4.3 in \cite{C14} for a modern treatment).
\end{itemize}
\end{defn}

\begin{lem} \label{lem:tits-to_weyl}
The restriction of $c$ to $N_{G_L}T_L$ induces an isomorphism $N_{G_L}T_L/T_L \simto W_L$.  Further restricting to $\tilde{W}_L$ induces an isomorphism $\tilde{W}_L/Hom(L,\mu_2) \simto W_L$.
\end{lem}
\begin{proof}
This follows from the fact that the Weyl group of an algebraic group is the subgroup of automorphisms of the character group generated by reflections in roots.
\end{proof}

\begin{lem} \label{lem:tits-embedding-outer}
Viewed as a subgroup of $N_{\Aut V_L} T_L$, $\tilde{W}_L$ is a normal subgroup of $O(\tilde{L})$, given by the preimage of $W_L$ in $O(L)$ under the $\Hom(L,\mu_2)$-cover $\pi: O(\tilde{L}) \to O(L)$.  In particular, the quotient $O(\tilde{L})/\tilde{W}_L$ is isomorphic to the outer automorphism group $O(L)/W_L$ of $L$.
\end{lem}
\begin{proof}
Lemma \ref{lem:tits-to_weyl} gives us the following commutative diagram of extensions, where the vertical arrows are inclusions of normal subgroups, and the squares on the top left and bottom right are pullbacks:

\[ \xymatrix{ 1 \ar[r] & \Hom(L,\mu_2) \ar[r] \ar[d] & \tilde{W}_L \ar[r] \ar[d] & W_L \ar@{=}[d] \ar[r] & 1\\
1 \ar[r] & T_L \ar[r] \ar@{=}[d]  & N_{G_L} T_L \ar[r] \ar[d] & W_L \ar[d] \ar[r] & 1 \\
1 \ar[r] & T_L \ar[r] & N_{\Aut V_L} T_L \ar[r] & O(L) \ar[r] & 1.} \]

By Lemma \ref{lem:OL-normalizes-torus}, $O(\tilde{L})$ normalizes $T_L$, i.e., we have an embedding $O(\tilde{L}) \hookrightarrow N_{\Aut V_L} T_L$.  By Proposition \ref{prop:structure-of-orthogonal-transformation-group}, restriction of $N_{\Aut V_L} T_L \to N_{\Aut V_L} T_L/T_L = O(L)$ to $O(\tilde{L})$ is a surjective homomorphism to $O(L)$, with kernel $\Hom(L,\mu_2)$.  We therefore have a commutative diagram of extensions, where the vertical arrows are inclusions of normal subgroups, and the left square is a pullback:

\[ \xymatrix{ 1 \ar[r] & \Hom(L,\mu_2) \ar[r] \ar[d]  & O(\tilde{L}) \ar[r] \ar[d] & O(L) \ar@{=}[d] \ar[r] & 1 \\
1 \ar[r] & T_L \ar[r] & N_{\Aut V_L} T_L \ar[r] & O(L) \ar[r] & 1.} \]

By the universal property of pullback, this diagram fits into the pullback diagram along the embeddings $N_{G_L} T_L \to N_{\Aut V_L} T_L$ and $O(\tilde{L}) \to N_{\Aut V_L} T_L$ as follows:

\[ \xymatrix @-1pc { 1 \ar[rr] & & \Hom(L,\mu_2) \ar[rr] \ar@{=}[dd]|\hole \ar[rd] & & \tilde{W}_L \ar[rr] \ar[dd]|\hole \ar[rd] & & W_L \ar[dd]|\hole \ar[rr] \ar@{=}[rd] & & 1 \\
& 1 \ar[rr] & & T_L \ar[rr] \ar@{=}[dd]  & & N_{G_L} T_L \ar[rr] \ar[dd] & & W_L \ar[dd] \ar[rr] & & 1 \\
1 \ar[rr] & & \Hom(L,\mu_2) \ar[rr] \ar[rd]  & & O(\tilde{L}) \ar[rr] \ar[rd] & & O(L) \ar@{=}[rd] \ar[rr] & & 1 \\
& 1 \ar[rr] & & T_L \ar[rr] & & N_{\Aut V_L} T_L \ar[rr] & & O(L) \ar[rr] & & 1,} \]
where the vertical and diagonal arrows are inclusions of normal subgroups, and the quadrilaterals without equal signs are pullbacks.  From the upper-right pullback rectangle, we obtain an isomorphism $O(\tilde{L})/\tilde{W}_L \simeq O(L)/W_L$.
\end{proof}

\subsection{Main Theorem}

\begin{thm} \label{thm:main}
Let $L$ be a positive definite even lattice, and let $V_L$ be the graded vertex algebra over $\bZ$ attached to $L$.  Then, the functor that takes an affine scheme $\Spec R$ to the group of homogeneous $R$-vertex algebra automorphisms of $V_{L,R}$ is represented by a finite type smooth affine group scheme $\Aut V_L$ over $\bZ$.  This group is given as a product $G_L O(\tilde{L})$, where the normal subgroup $G_L$ is the split reductive group given in Definition \ref{defn:GL}, and $O(\tilde{L})$ is the finite flat group scheme defined in Proposition \ref{prop:structure-of-orthogonal-transformation-group}.  The intersection of $G_L$ and $O(\tilde{L})$ is the Tits subgroup of the normalizer of the torus $T_L$ in $G_L$, i.e., the canonical elementary 2-torsion extension of the Weyl group.  The quotient $\Aut V_L/G_L$ is isomorphic to the outer automorphism group $\Aut L/W_L$ of $L$.  Furthermore, the functor that takes an affine scheme $\Spec R$, where $\det L$ is invertible in $R$, to the group of $R$-vertex operator algebra automorphisms of $V_{L,R}$ is represented by the base change $(\Aut V_L)_{\bZ[1/\det L]}$.
\end{thm}
\begin{proof}
The finite type affine property over $\Spec \bZ$ is Proposition \ref{prop:aut-represented-by-affine-finite-type}.  The normality of $G_L$ is Proposition \ref{prop:normality}.  Smoothness follows from the fact that $\Aut V_L/G_L$ is a quotient of the constant group scheme $O(L)$.

The product property follows from the following sequence of reductions: By Corollary \ref{cor:product-with-normalizer}, $\Aut V_L = G_L N_{\Aut V_L} T_L$, but by combining Propositions \ref{prop:normalizer-as-product} and \ref{prop:centralizer}, we have $N_{\Aut V_L} T_L = T_L O(\tilde{L})$, and the $T_L$ can be absorbed into the $G_L$.

We now show that the intersection of $G_L$ and $O(\tilde{L})$ is the Tits subgroup $\tilde{W}_L$.  By Lemma \ref{lem:tits-embedding-outer}, $\tilde{W}_L$ is contained in the intersection, so it suffices to show the reverse inclusion.  Since $O(\tilde{L})$ normalizes $T_L$, the intersection is contained in $N_{G_L}(T_L)$.  Taking the quotient by $T_L$, we find that the subgroup $O(\tilde{L})$ is mapped to $O(L)$, and the intersection is taken to $W_L$.  Thus, the intersection is contained in the pullback of $W_L$ in $O(\tilde{L})$, which is precisely the Tits group by Lemma \ref{lem:tits-embedding-outer}.

From the product property and the ``second isomorphism theorem'' for groups, the quotient $\Aut V_L/G_L = O(\tilde{L})G_L/G_L$ is isomorphic to the quotient of $O(\tilde{L})$ by the intersection $\tilde{W}_L$, and by Lemma \ref{lem:tits-embedding-outer}, this is the outer automorphism group of $L$.

The last claim follows immediately from Corollary \ref{cor:homogeneous-aut-VOA-aut}.
\end{proof}

\section{Additional questions}

It is natural to ask how these results can be generalized:
\begin{enumerate}
\item If we remove the evenness condition on $L$, there is an analogous construction of a vertex operator superalgebra $V_L$, and its automorphisms form a super group scheme. It would be nice to have a characterization of the groups we get.
\item If we remove the integrality condition from $L$, there is still an analogous construction of an abelian intertwining algebra \cite{DL93}.  Once again, we can ask about the symmetries of this object.
\item If we remove the condition that $L$ is positive definite, we obtain a conformal vertex algebra whose infinitesimal symmetries typically form an infinite dimensional Lie algebra that is highly non-integrable.  The standard example is the even unimodular lorentzian lattice $I\!I_{1,25}$.  The symmetries of the corresponding conformal vertex algebra have a lot of structure: The discrete symmetries of $V_{I\!I_{1,25}}$ include a $\mu_2^{26}$ cover of the automorphism group of the lattice, and the fake monster Lie algebra \cite{B90} is a subquotient of the infinitesimal symmetries.  The automorphism group of the lattice contains a huge reflection group whose Dynkin diagram is the Leech lattice $\Lambda$, and its outer automorphism group is $Co_\infty = \Lambda\cdot Co_0$ \cite{C83}.
\item If we remove the condition that symmetries must preserve the $\bZ$-grading, we do not get new automorphisms over $\mathbb{R}$ (see \cite{M18} Theorem 7.1), but over other rings, this question has not been studied.
\item It is of course natural to move beyond lattices, since many other vertex algebras that admit forms over small rings have interesting symmetry groups.  The basic example is $V^\natural_\bZ$, constructed in \cite{C17}, where the automorphism group scheme is the constant monster group.  A non-constant physically relevant example is, $V^\natural_{\bZ} \otimes V_{I\!I_{1,1}}$, which has discrete symmetries that include the monster simple group, algebraic group symmetries that include $GL_2$, and infinitesimal symmetries that produce the monster Lie algebra as a subquotient.  Examples related to affine Kac-Moody Lie algebras and W-algebras may be especially interesting.
\end{enumerate}

\end{document}